\documentclass[12pt,letterpaper]{amsart}
\usepackage{amsmath,amssymb}
\usepackage{amsthm}
\usepackage[abbrev]{amsrefs}
\usepackage{latexsym}
\usepackage{hyperref}
\usepackage{color}

\numberwithin{equation}{section}
\theoremstyle{plain}
\newtheorem{thm}{Theorem}[section]
\newtheorem{prop}[thm]{Proposition}
\newtheorem{lem}[thm]{Lemma}
\newtheorem{cor}[thm]{Corollary}

\theoremstyle{definition}
\newtheorem{dfn}[thm]{Definition}

\theoremstyle{remark}
\newtheorem{rem}[thm]{Remark}
\newtheorem*{ack}{Acknowledgment}

\newcommand{\set}[1]{\{\,{#1}\,\}}
\newcommand{\R}{\mathbb R}
\newcommand{\N}{\mathbb N}

\DeclareMathOperator{\dis}{dis}
\DeclareMathOperator{\supp}{supp}
\DeclareMathOperator{\id}{id}
\DeclareMathOperator{\pr}{pr}
\DeclareMathOperator{\diam}{diam}
\DeclareMathOperator{\dconc}{{\it d}_{{\rm conc}}}

\title{Extension of Gromov's Lipschitz order to with additive errors}
\author{Hiroki Nakajima}
\address{Mathematical Sciences Course, Ehime University, Matsuyama 790-8577, JAPAN}
\email{nakajima.hiroki.nz@ehime-u.ac.jp}
\date{\today}
\keywords{metric measure space, box distance, transport plans, couplings, Gromov-Prokhorov distance, Lipschitz order}
\thanks{The author is supported by JSPS KAKENHI Grant Number 22K13908}
\begin{document}
\maketitle
\begin{abstract}
Gromov's Lipschitz order is an order relation on the set of metric measure spaces. One of the compactifications of the space of isomorphism classes of metric measure spaces equipped with the concentration topology is constructed by using the Lipschitz order. The concentration topology is deeply related to the concentration of measure phenomenon.
In this paper, we extend the Lipschitz order to that with additive errors and prove useful properties. We also discuss the relation of it to a map with the property of $1$-Lipschitz up to an additive error.
\end{abstract}
\tableofcontents
\section{Introduction}
In some high-dimensional space, it is observed that any 1-Lipschitz function is close to a constant function. That principle is called the concentration of measure phenomenon (see \cite{Led:conc}\cite{Levy:conc}\cite{MilV:conc}). The observable metric $\dconc$ on the space, say $\mathcal X$, of isomorphism classes of metric measure spaces is introduced by M. Gromov based on the concentration of measure phenomenon in \cite{Gmv:green}. The compactification of $(\mathcal X,\dconc)$ is also constructed by him in \cite{Gmv:green}. The box metric $\square$ (Definition \ref{dfn:boxOptimal}) is a metric on $\mathcal X$, which is more elementary than the observable metric $\dconc$. The Lipschitz order $\prec$ (Definition \ref{def:Lip_ord})  is a partial order on $\mathcal X$. The box metric $\square$ and the Lipschitz order $\prec$ are important because the compactification of $(\mathcal X,\dconc)$ in \cite{Gmv:green} is constructed by these concepts.
Note that the box metric $\square$ is equivalent to the Gromov-Prokhorov metric \cite{Loh:dGP}.
The topology given by $\dconc$ is strictly weaker than the topology given by $\square$. It is surprising that the compactification of the metric space $(\mathcal X,\dconc)$ is constructed by using more elementary metric space $(\mathcal X,\square)$.

When considering the compactification of the $(\mathcal X,\dconc)$, we need to know the relation between the Lipschitz order $\succ$ and the box metric $\square$. In this paper, we introduce an extension of the Lipschitz order to that with additive errors and extend the following theorem.
\begin{thm}[Theorem 4.35 in \cite{Shioya:mmg}]\label{thm:box-stable_noError}
Let $X$, $Y$, $X_n$, and $Y_n$ be mm-spaces, $n=1,2,\dots$. We assume that $X_n$ and $Y_n$ $\square$-converge to $X$ and $Y$, respectively.
If $X_n\succ Y_n$ for any positive integer $n$, then we have $X\succ Y$.
\end{thm}

For Borel probability measures on $\R$, the Lipschitz order with additive errors is defined by the author \cite{Nkj:isop}. It is applied to the study of isoperimetric inequalities on metric measure spaces.

\begin{dfn}[Lipschitz order with additive errors]\label{dfn:lipE}
Let $X$ and $Y$ be two mm-spaces, and $\varepsilon\ge 0$ a real number. We define 
\[
\dis_\succ S:=\sup\{d_Y(y,y')-d_X(x,x')\mid (x,y),(x',y')\in S\}
\]
for non-empty closed set $S\subset X\times Y$ and $\dis \emptyset :=0$.
We say that {\it $X$ dominates $Y$ with an additive error $\varepsilon$} and write $X\succ_\varepsilon Y$ if there exist a coupling $\pi\in\Pi(X,Y)$ and a closed set $S\subset X\times Y$ such that $\dis_\succ S\le\varepsilon$ and $\pi(S)\ge 1-\varepsilon$. 
\end{dfn}
We remark that the notation $X\succ_\varepsilon Y$ is also defined in Definition 3.3 in \cite{KY:bdd1} but that is different from Definition \ref{dfn:lipE}. The relation between them is explained in Section \ref{sec:upto}.
The following are key properties of the Lipschitz order with additive errors.
\begin{thm}\label{thm:Lip_main}
Let $X$, $Y$, and $Z$ be three mm-spaces, and $s,t\ge 0$ two real numbers. Then we have the following $(1)$, $(2)$, and $(3)$.
\begin{enumerate}
\item $X\succ Y$ if and only if $X\succ_0 Y$.\label{thm:Lip_main:eq1}
\item $X\succ_{t+\varepsilon} Y$ for any $\varepsilon >0$ if and only if $X\succ_t Y$.\label{thm:Lip_main:eq2}
\item If $X\succ_s Y$ and $Y\succ_t Z$, then $X\succ_{s+t} Z$.\label{thm:Lip_main:eq3}
\end{enumerate}
\end{thm}
A basic relation between the box metric $\square$ and the Lipschitz order with additive errors is as follows.
\begin{prop}\label{prop:boxToLip}
Let $X$ and $Y$ be two mm-spaces. If $\square(X,Y)\le\varepsilon$, then we have $X\succ_\varepsilon Y$ and $Y\succ_\varepsilon X$.
\end{prop}
Proposition \ref{prop:boxToLip} follows from Proposition \ref{prop:dis_le} and Definition \ref{dfn:boxOptimal}. 
The following Theorem \ref{thm:box-stable} is an extension of Theorem \ref{thm:box-stable_noError}. In fact, Theorem \ref{thm:box-stable} implies Theorem \ref{thm:box-stable_noError} if $\varepsilon_n=\varepsilon=0$. The original proof of Theorem \ref{thm:box-stable_noError} is complicated but the proof of Theorem \ref{thm:box-stable} is much  clearer by owing to Definition \ref{dfn:lipE}.
\begin{thm}\label{thm:box-stable}
Let $X$, $Y$, $X_n$, and $Y_n$ be mm-spaces, $n=1,2,\dots$. We assume that $X_n$ and $Y_n$ $\square$-converge to $X$ and $Y$, respectively.
Let $\varepsilon$ and $\varepsilon_n$ be non-negative real numbers, $n=1,2,\dots$. We assume that $\varepsilon_n$ converges to $\varepsilon$. If  $X_n\succ_{\varepsilon_n} Y_n$ for any positive integer $n$, then we have $X\succ_\varepsilon Y$.
\end{thm}
Note that in the case $\varepsilon=0$, Theorem \ref{thm:box-stable} implies  Lemma 3.13 in \cite{KY:bdd1}. We remark that if we adopt Definition 3.3 in  \cite{KY:bdd1} as the definition of $X\succ_\varepsilon Y$, then Theorem \ref{thm:box-stable} does not hold in the case $\varepsilon> 0$ (see Remark \ref{rem:lsCounter}).
The proof of Theorem \ref{thm:box-stable} follows from Theorem \ref{thm:Lip_main} as follows.
\begin{proof}[Proof of Theorem \ref{thm:box-stable}]
We assume that $X_n\succ_{\varepsilon_n} Y_n$. Put 
$s_n:=\square(X_n,X)$ and $t_n:=\square(Y_n,Y)$. By Proposition \ref{prop:boxToLip}, we have
$X\succ_{s_n} X_n$ and $Y_n\succ_{t_n}Y$. By Theorem \ref{thm:Lip_main} \eqref{thm:Lip_main:eq3}, we have $X\succ_{s_n+t_n+\varepsilon_n} Y$. By Theorem \ref{thm:Lip_main} \eqref{thm:Lip_main:eq2}, we have $X\succ_\varepsilon Y$ because $s_n+t_n+\varepsilon_n\to \varepsilon$. This completes the proof.
\end{proof}

We define the {\it unilateral box metric} $\square_\succ\colon \mathcal X\times \mathcal X\to[0,\infty)$ by
\[
\square_\succ(X,Y):=\inf\{\varepsilon \ge 0\mid X\succ_\varepsilon Y\} \quad\text{for $X,Y\in \mathcal X$,}
\]
then we obtain the following Corollary \ref{cor:Lip_main1} as a corollary of Theorem \ref{thm:Lip_main}.
\begin{cor}\label{cor:Lip_main1}
Let $X$, $Y$, and $Z$ be three mm-spaces. Then we have the following $(1)$, $(2)$, and $(3)$.
\begin{enumerate}
\item $X\succ Y$ if and only if $\square_\succ(X,Y)=0$.
\item $\square_\succ(X,Z)\le\square_\succ(X,Y)+\square_\succ(Y,Z)$.
\end{enumerate}
\end{cor}
Note that the unilateral box metric is not a metric. We remark that Proposition \ref{prop:boxToLip} implies
\[
\square_\succ(X,Y)\le \square(X,Y)\quad\text{for $X,Y\in\mathcal X$}
\]
and Theorem \ref{thm:box-stable} implies
\[
\square_\succ(\lim_{n\to \infty}X_n,\lim_{n\to \infty}Y_n)\le\liminf_{n\to\infty}\square_\succ(X_n,Y_n)
\]
for any two $\square$-convergent sequences $(X_n)_n$ and $(Y_n)_n$.

In Section \ref{sec:main}, we give a proof of Theorem \ref{thm:Lip_main} which is basic properties of the Lipschitz order with additive errors.
In Section \ref{sec:upto}, we explain the relation between the Lipschitz order with an additive error $\varepsilon$ and $1$-Lipschitz continuity up to $\varepsilon$ (Definition \ref{dfn:upto}).

\begin{ack}
The author would like to thank Daisuke Kazukawa and Takashi Shioya for very valuable
comments.
\end{ack}

\section{Preliminary}
In this section, we present some basics of mm-space.
We refer to \cites{Gmv:green,Shioya:mmg} for more details about the contents of this section. We denote by $\overline{A}$ the topological closure of $A$ for a subset $A$ of a topological space.  
%\subsection{Some basics of \mmsp}\label{basics}
\begin{dfn}[mm-space]
Let $(X,d_X)$ be a complete separable metric space and $m_X$ a Borel probability measure on $X$.
We call such a triple $(X,d_X,m_X)$ an {\it mm-space}.
We sometimes say that $X$ is an mm-space, for which the metric and measure of $X$ are respectively indicated by $d_X$ and $m_X$. 
%We put $tX:=(X, td_X,m_X)$ for $t>0$. Since an mm-space is equipped with the probability measure, it is nonempty.
\end{dfn}

We denote the Borel $\sigma$-algebra over $X$ by $\mathcal B_X$.
For any point $x\in X$, any two subsets $A,B\subset X$, and any real number $r\ge 0$, we define
\begin{align*}
d_X(x,A)&:=\inf\set{d_X(x,y)\mid y\in A} ,\\
%d_X(A,B)&:=\inf\set{d_X(x,y)\mid x\in A,\,y\in B},\\
U_r^{d_X}(A)&:=\set{y\in X\mid d_X(y,A)<r},\\
B_r^{d_X}(A)&:=\set{y\in X\mid d_X(y,A)\leq r},
\end{align*}
where $\inf \emptyset:=\infty$. The symbol $U_r^{d_X}(A)$ and $B_r^{d_X}(A)$ are often omitted as $U_r(A)$ and $B_r(A)$, respectively. We remark that $U_r(\emptyset)=B_r(\emptyset)=\emptyset$ for any real number $r\ge 0$.  The {\it diameter of $A$} is defined by $\diam A:=\sup_{x,y\in A}d_X(x,y)$ for $A\neq\emptyset$ and $\diam \emptyset:=0$.

Let $Y$ be a topological space and let $p:X\to Y$ be a measurable map from a measure space $(X,m_X)$ to a measurable space $(Y,\mathcal B_Y)$. {\it The push-forward $p_*m_X$ of $m_X$ by the map $p$} is defined as $p_*m_X(A):=m_X(p^{-1}(A))$ for any $A\in \mathcal B_Y$.

\begin{dfn}[support]
Let $\mu$ be a Borel measure on a topological space $X$.
We define the {\it support} $\supp\mu$ of $\mu$ by
\[
\supp\mu:=\{x\in X\mid \text{ $m_X(U)>0$ for any open neighborhood $U$ of $x$} \}.
\]
\end{dfn}
\begin{prop}[Proposition 2.3 in \cite{Nkj:optimal}]\label{prop:push}
Let $X$ and $Y$ be two topological spaces and let $f\colon X\to Y$ be a continuous map. If a Borel measure $\mu$ on $X$ satisfies 
\[
\mu(X\setminus \supp\mu)=0,
\]
then we have
\[
\supp f_*\mu=\overline{f(\supp\mu)}.
\]
\end{prop}
\begin{proof}
Since
\begin{align*}
f_*\mu(Y\setminus \overline{f(\supp\mu)})&=\mu(X\setminus f^{-1}(\overline{f(\supp\mu)}))\\
&\le \mu(X\setminus \supp\mu)=0,
\end{align*}
we have $\supp f_*\mu\subset \overline{f(\supp\mu)}$.

Next, let us prove
\begin{equation}\label{prop:push:goal}
f(\supp \mu)\subset \supp f_*\mu.
\end{equation}
Take any $y\in f(\supp\mu)$ and any open neighborhood $U$ of $y$.
Since $y\in f(\supp\mu)$, there exists $x\in\supp\mu$ such that $y=f(x)$.
The set $f^{-1}(U)$ is an open neighborhood of $x\in\supp\mu$ because $f$ is continuous. Then we have
$
f_*\mu(U)= \mu(f^{-1}(U))>0.
$
This implies that $y\in\supp f_*\mu$ and we obtain \eqref{prop:push:goal}.
Since the set $\supp f_*\mu$ is closed, we have
\[
\overline{f(\supp\mu)}\subset \supp f_*\mu.
\]
This completes the proof.
\end{proof}
\begin{rem}
Let $\mu$ be a Borel measure $\mu$ on a topological space $X$. We have $\mu(X\setminus \supp \mu)=0$ if $X$ is second countable. Even if we assume that $\mu$ is inner regular, the equation $\mu(X\setminus \supp \mu)=0$ holds.
\end{rem}
\begin{dfn}[mm-isomorphism]\label{dfn:mmIso}
Two mm-spaces $X$ and $Y$ are said to be {\it mm-isomorphic} if there exists an isometry $f:\supp m_X\to\supp m_Y$ such that $f_*m_X=m_Y$.
Such an isometry $f$ is called an {\it mm-isomorphism}.
The mm-isomorphism relation is an equivalence relation on the class of mm-spaces.
Denote by $\mathcal X$ the set of mm-isomorphism classes of mm-spaces.
\end{dfn}
%Note that $X$ is mm-isomorphic to $(\supp\mux,\dx,\mux)$. We assume that any \mmsp\ $X$ satisfies \[X=\supp \mux\] unless otherwise stated.
%We give the definition of the Lipschitz order. We consider the maximum and maximal elements of the 1-measurement of an mm-space with respect to this order relation.
\begin{dfn}[Lipschitz order]\label{def:Lip_ord}
Let $X$ and $Y$ be two mm-spaces.
We say that $X$ {\it dominates} $Y$ and write $Y \prec X$ if there exists a 1-Lipschitz map $f:\supp m_X\to \supp m_Y$ satisfying
\[
f_*m_X=m_Y.
\]
We call the relation $\prec$ on $\mathcal X$ the {\it Lipschitz order}.
\end{dfn}
\begin{prop}[Proposition 2.11 in \cite{Shioya:mmg}]\label{prop:lipPartialOrder}
The Lipschitz order $\prec$ is a partial order relation on $\mathcal X$.
\end{prop}

\begin{dfn}[Transport plan]\label{dfn:trans}
Let $\mu$ and $\nu$ be two Borel probability measures on two topological spaces $X$ and $Y$, respectively. We say that a Borel probability measure $\pi$ on $X\times Y$ is a transport plan between $\mu$ and $\nu$ if we have $(\pr_1)_*\pi=\mu$ and $(\pr_2)_*\pi=\nu$, where $\pr_1\colon X\times Y\to X$ and $\pr_2\colon X\times Y\to Y$ are the first and second projections, respectively. We denote by $\Pi(\mu,\nu)$ the set of transport plans between $\mu$ and $\nu$.
\end{dfn}
\begin{prop}\label{prop:suppPi}
Let $X$ and $Y$ be two mm-spaces. Then we have
\[
\supp\pi\subset (\supp m_X)\times (\supp m_Y)
\]
for any $\pi\in\Pi(m_X,m_Y)$.
\end{prop}
\begin{proof}
Take any $(x,y)\in\supp\pi$. By Proposition \ref{prop:push}, we have
\[
x\in \pr_1(\supp\pi)\subset \supp((\pr_1)_*\pi)=\supp m_X.
\]
Similarly, we have $y\in\supp m_X$.
\end{proof}

\begin{prop}[Lemma 7.1 in \cite{Nkj:optimal}]\label{prop:nonDeg}
Let $X$ and $Y$ be two mm-spaces and $\pi$ a transport plan between $m_X$ and $m_Y$. We assume that a map $f\colon \supp m_X\to \supp m_Y$ satisfies
\[
\supp\pi\subset \{(x,f(x))\mid x\in\supp m_X\}.
\]
Then we have $f_*m_X=m_Y$.
\end{prop}
Let $(X,d_X)$ and $(Y,d_Y)$ be two metric spaces. We define 
\[
\dis S:=\sup\{|d_X(x,x')-d_Y(y,y')|: (x,y),(x',y')\in S\}
\]
for non-empty set $S\subset X\times Y$ and $\dis\emptyset:=0$.
\begin{prop}\label{prop:dis_le}
Let $(X,d_X)$ and $(Y,d_Y)$ be two metric spaces. Then we have
\[
\dis_{\succ} S\le \dis S
\]
for $S\subset X\times Y$, where $\dis_\succ S$ is in Definition \ref{dfn:lipE}.
\end{prop}
\begin{proof}
The proof is easy and omitted.
\end{proof}
\begin{dfn}[Box metric]\label{dfn:boxOptimal}
Let $X$ and $Y$ be two mm-spaces, then we have
\[
\square(X,Y)=\inf_{S\subset X\times Y}\inf_{\pi\in \Pi(m_X,m_Y)}\max\{\dis S, 1-\pi(S)\},
\]
where $S\subset X\times Y$ runs over all closed subsets of $X\times Y$.
\end{dfn}
\begin{rem}
Definition \ref{dfn:boxOptimal} is not the original definition of the box metric $\square$. The original definition is in Section $3\frac12.3$ in \cite{Gmv:green} (cf. Definition 4.4 in \cite{Shioya:mmg}). The equivalence between the original definition and Definition \ref{dfn:boxOptimal} is proved in Theorem 1.1 in \cite{Nkj:optimal}.
\end{rem}
%\begin{thm}[Theorem 4.10 in \cite{Shioya:mmg}]
%The function $\square$ is a metric on the set $\mathcal X$ of mm-isomorphism classes of \mmsp s.
%\end{thm}
\begin{dfn}[Prokhorov metric]
The Prokhorov metric $d_P$ is defined by
\[
d_P(\mu,\nu):=\inf\set{\varepsilon>0\mid \mu(U_\varepsilon(A))\ge \nu(A)-\varepsilon \text{ for any Borel set $A\subset X$}}
\]
for any two Borel probability measures $\mu$ and $\nu$ on a metric space $X$.
\end{dfn}
\begin{prop}[Proposition 4.12 in \cite{Shioya:mmg}]\label{prop:box_prok}
For any two Borel probability measures $\mu$ and $\nu$ on a complete separable metric space $X$, we have
\[
\square((X,\mu),(X,\nu))\leq 2d_P(\mu,\nu).
\]
\end{prop}
%\begin{dfn}[Ky Fan metric]
%Let $(X,\mu)$ be a measure space. For two $\mu$-measurable maps $f,g:X\to\R$, we define the {\it Ky Fan metric} $d_{\rm KF}=d_{\rm KF}^\mu$ by
%\[
%d_{\rm KF}^\mu(f,g):=\inf\set{\varepsilon\ge 0\mid \mu(\set{x\in X\mid |f(x)-g(x)|>\varepsilon})\le\varepsilon}.
%\]
%\end{dfn}

%\begin{lem}[Lemma 1.26 in \cite{Shioya:mmg}]
%Let $(X,\mu)$ be a measure space. For two $\mu$-measurable maps $f,g:X\to\R$, we have
%\[
%\dP(f_*\mu,g_*\mu)\le \dkf\mu(f,g).
%\]
%\end{lem}

%\begin{dfn}[Observable distance]
%For a parameter $\varphi$ of an mm-space $X$, we define 
%\[
%\Lip_1(X):=\set{f:X\to \R\mid \text{ $f$ is $1$-Lipschitz}}
%\]
%and 
%\[
%\varphi^*\Lip_1(X):=\set{f\circ \varphi\mid f\in \Lip_1(X)}.
%\]
%The Hausdorff distance function $d_{\rm H}^{\rm KF}$ is defined with respect to $d_{\rm KF}^{\mathcal L^1}$.
%We define the {\it observable distance} $\dconc$ between two mm-spaces $X$ and $Y$ by
%\[
%\dconc(X,Y):=\inf_{\varphi,\psi} d_{\rm H}^{\rm KF}(\varphi^*\Lip_1(X),\psi^*\Lip_1(Y)),
%\]
%where $\varphi:I:=[0,1)\to X$ and $\psi:I\to Y$ run over all parameters of $X$ and $Y$ respectively.
%\end{dfn}
%\begin{thm}[Theorem 5.13 in \cite{Shioya:mmg}]
%The function $\dconc$ is a metric on $\cX$.
%\end{thm}
%\begin{prop}[Proposition 5.5 in \cite{Shioya:mmg}]\label{prop:concBox}
%For two mm-spaces $X$ and $Y$, we have
%$\dconc(X,Y)\le \square (X,Y)$.
%\end{prop}
Let us explain the weak Hausdorff convergence which is used in the proof of Theorem \ref{thm:Lip_main} \eqref{thm:Lip_main:eq2}.
\begin{dfn}[Weak Hausdorff convergence]\label{dfn:wH}
Let $(X,d_X)$ be a metric space. Let $A,A_n\subset X$, $n=1,2,\dots$, be closed subsets of $X$. We say that {\it $A_n$ converges weakly to $A$ as $n\to\infty$} if the following \eqref{dfn:wH1} and \eqref{dfn:wH2} are both satisfied.
\begin{enumerate}
\item For any $x\in A$, we have
\[
\lim_{n\to\infty}d_X(x,A_n)=0.
\]\label{dfn:wH1}
\item For any $x\in X\setminus A$, we have
\[
\liminf_{n\to\infty}d_X(x,A_n)>0.
\]\label{dfn:wH2}
\end{enumerate}
This convergence is called the {\it weak Hausdorff convergence} or the {\it Kuratowski-Painlev\'e convergence}. 
\end{dfn}

For the weak Hausdorff convergence, the following Theorem \ref{thm:KPcpt} is useful.
\begin{thm}[Theorem 5.2.12 in \cite{Beer:Top}]\label{thm:KPcpt}
Let $X$ be a second countable metric space. For any sequence of closed subsets of $X$, there exists a convergent subsequence in the weak Hausdorff sense.
\end{thm}

\begin{lem}[Lemma 6.4 in \cite{NkjShioya:mm-gp}]\label{lem:KP_ProkLsConti}
Let $X$ be a complete separable metric space, $S, S_n\subset X$ closed subsets, and $\mu, \mu_n$
 Borel probability measures, $n=1,2,\dots$. If $\mu_n$ converges weakly to $\mu$ and if $S_n$ converges to $S$ in the weak Hausdorff, then
 \begin{equation}\label{lem:KP_ProkLsConti:eqMain}
 \mu(S)\ge\limsup_{n\to \infty}\mu_n(S_n).
 \end{equation}
 \end{lem}
\begin{proof}
In general, we have
\[
\{x\in X\mid \liminf_{n\to\infty} d_X(x,S_n)=0\}=\bigcap_{t>0}\bigcap_{n\in\N}\bigcup_{k\ge n} U_t(S_k).
\]
Since $S_n$ converges to $S$ in the weak Hausdorff, we have
\[
S\supset \{x\in X\mid \liminf_{n\to\infty} d_X(x,S_n)=0\}
\]
and these imply
\begin{equation}\label{lem:KP_ProkLsConti:eq1}
\mu(S)\ge \lim_{t\to +0}\limsup_{n\to\infty}\mu\left(U_t(S_n)\right).
\end{equation}
Now, let us prove \eqref{lem:KP_ProkLsConti:eqMain}. Take any $\varepsilon>0$. By \eqref{lem:KP_ProkLsConti:eq1}, there exist $\delta_1>0$ and $N_1\in\N$ such that
\begin{equation}\label{lem:KP_ProkLsConti:eq2}
\mu(U_t(S_n))<\mu(S)+\varepsilon
\end{equation}
for any $t\in (0,\delta_1]$ and $n\ge N_1$.
Put $\delta_2:=\min\{\varepsilon,\delta_1\}$. Since $X$ is separable and $\mu_n$ converges to $\mu$ weakly, there exists $N_2$ such that
\begin{equation}\label{lem:KP_ProkLsConti:eq3}
d_P(\mu_n,\mu)<\delta_2
\end{equation}
for any $n\ge N_2$. Take any $n\ge \max\{N_1,N_2\}$.
By \eqref{lem:KP_ProkLsConti:eq2} and \eqref{lem:KP_ProkLsConti:eq3}, we obtain
\[
\mu(S)+\varepsilon > \mu(U_{\delta_2}(S_n))\ge\mu_n(S_n)-\varepsilon.
\]
This completes the proof.
\end{proof}
\begin{rem}
The proof of Lemma \ref{lem:KP_ProkLsConti} is more directly than the proof of Lemma 6.4 in \cite{NkjShioya:mm-gp}.
\end{rem}

\section{Lipschitz order with additive errors}\label{sec:main}
In this section, we give a proof of Theorem \ref{thm:Lip_main}.
\begin{proof}[Proof of Theorem \ref{thm:Lip_main} \eqref{thm:Lip_main:eq1}]
We assume that $X\succ Y$. Let us prove $X\succ_0 Y$. Since $X\succ Y$, there exists a $1$-Lipschitz map $f\colon \supp m_X\to \supp m_Y$ such that $f_*m_X=m_Y$. Put 
\[\pi:=(\id_{\supp m_X},f)_*m_X\in \Pi(m_X,m_Y)\]
and $S:=\supp \pi$. Then we have $\dis_{\succ} S=0$ and $1-\pi(S)=0$. This implies that $X\succ_0 Y$.

Conversely, we assume that $X\succ_0 Y$.  There exist $\pi\in \Pi(m_X,m_Y)$ and a closed set $S\subset X\times Y$ such that $\dis_\succ S=0$ and $1-\pi(S)=0$. Since $\pi(S)=1$, we have $\supp\pi\subset S$. This implies that $\dis_\succ(\supp\pi)=0$.

Now, there exists a unique map $f\colon \supp m_X\to \supp m_Y$ such that $(x,f(x))\in \supp\pi$ for any $x\in \supp m_X$. Let us prove the existence of $f$. Take any point $x\in \supp m_X$. By Proposition \ref{prop:push}, we have 
\[\supp m_X=\supp (\pr_1)_*\pi=\overline{(\pr_1)(\supp\pi)}.\]
Then there exists $((x_n,y_n))_{n\in\N}\in (\supp\pi)^\N$ such that $(x_n)_{n\in\N}$ converges to $x$, where $\N$ is the set of positive integers. Take any positive integers $m$ and $n$. The conditions $(x_n,y_n)\in\supp\pi$ and $\dis_\succ \supp\pi=0$ show that
\[
d_Y(y_m,y_n)\le d_X(x_m,x_n).
\]
Then the sequence $(y_n)_{n\in \N}$ is Cauchy because the sequence $(x_n)_{n\in\N}$ is Cauchy. By the completeness of $Y$, there exists $y\in Y$ such that $(y_n)_{n\in \N}$ converges to $y$. Since $\supp\pi\subset X\times Y$ is closed, we have $(x,y)\in \supp\pi$. By Proposition \ref{prop:suppPi}, we have $y\in \supp m_Y$. Therefore we put $f(x):=y$. Next, let us prove the uniqueness of $f$. Take any point $x\in \supp m_X$ and any points $y,y'\in \supp m_Y$ with $(x,y),(x,y')\in \supp\pi$. Since $\dis_\succ\supp\pi=0$,
\[
d_Y(y,y')\le d_X(x,x)=0.
\]

The map $f$ is $1$-Lipschitz because $\dis_\succ\supp\pi=0$. The uniqueness of $f$ implies that
\[
\supp\pi =\{(x,f(x))\mid x\in\supp m_X\}.
\]
This shows that $f_*m_X=m_Y$ by Proposition \ref{prop:nonDeg}. This completes the proof.
\end{proof}

\begin{lem}\label{lem:disSle}
Let $X$ and $Y$ be two metric spaces. Let $X\times Y$ be the product space equipped with the $l^1$-metric. Then we have
\[
\dis_\succ U_t(S)\le \dis_\succ S +2t
\]
for any $t>0$ and $S\subset X\times Y$.
\end{lem}
\begin{proof}
Let $S\subset X\times Y$ and  $t>0$. Take any two points $(x_1,y_1),(x_2,y_2)\in U_t(S)$. It is sufficient to prove that
\[
d_X(x_1,x_2)-d_Y(y_1,y_2)\le \dis_\succ S+2t.
\]
By $(x_i,y_i)\in U_t(S)$, there exists $(x_i',y_i')\in S$ such that
\[
d_X(x_i,x_i')+d_Y(y_i,y_i')<t
\]
for $i=1,2$. Now we have
\begin{align*}
&d_X(x_1,x_2)-d_Y(y_1,y_2)\\
&\le d_X(x_1',x_2')-d_Y(y_1',y_2')+\sum_{i=1}^2\{d_X(x_i,x_i')+d_Y(y_i,y_i')\}\\
&<\dis_\succ S+2t.
\end{align*}
This completes the proof.
\end{proof}
\begin{lem}[cf. \cite{Nkj:optimal}]\label{lem:DisLsConti}
Let $X$ and $Y$ be two metric spaces. Let $X\times Y$ be the product space equipped with the $l^1$-metric $d_{X\times Y}$. Let $S,S_n\subset X\times Y$, $n=1,2,\dots,$ be closed subsets. If $S_n$ converges weakly to $S$ as $n\to \infty$, then we have
\[
\dis_\succ S\le \liminf_{n\to \infty} \dis_\succ S_n.
\]
\end{lem}
\begin{proof}
Since $S_n$ converges weakly to $S$ as $n\to \infty$, we have
\begin{align*}
S&\subset \{(x,y)\in X\times Y\mid \lim_{n\to \infty} d_{X\times Y}((x,y),S_n)=0\}\\
&=\bigcap_{\varepsilon >0}\bigcup_{n}\bigcap_{k\ge n}U_\varepsilon (S_k).
\end{align*}
By Lemma \ref{lem:disSle}, we have
\begin{align*}
\dis_\succ S\le & \dis_\succ \left(\bigcap_{\varepsilon >0}\bigcup_{n}\bigcap_{k\ge n}U_\varepsilon (S_k)\right)\\
&\le \lim_{\varepsilon\to +0}\dis_\succ \left(\bigcup_{n}\bigcap_{k\ge n}U_\varepsilon (S_k)\right)\\
&= \lim_{\varepsilon\to +0}\sup_n\dis_\succ \left(\bigcap_{k\ge n}U_\varepsilon (S_k)\right)\\
&\le \lim_{\varepsilon\to +0}\sup_n\inf_{k\ge n}\dis_\succ \left(U_\varepsilon (S_k)\right)\\
&\le \lim_{\varepsilon\to +0}\liminf_{n\to \infty}(\dis_\succ S_n+2\varepsilon)\\
&=\liminf_{n\to \infty}\dis_\succ S_n.
\end{align*}
\end{proof}

\begin{proof}[Proof of Theorem \ref{thm:Lip_main} \eqref{thm:Lip_main:eq2}]
Take any positive integer $n$ and assume that $X\succ_{t+\frac 1n}Y$. By Definition \ref{dfn:lipE}, there exist $\pi_n\in \Pi(m_X,m_Y)$ and a closed set $S_n\subset X\times Y$ such that $\dis_\succ S_n\le t+\frac 1n$ and $1-\pi_n(S_n)\le t+\frac 1n$.
By Theorem \ref{thm:KPcpt}, there exists a closed set $S\subset X\times Y$ such that $S_n$ weakly converges to $S$ by taking a subsequence. Since $\Pi(m_X,m_Y)$ is compact with respect to the weak convergence, we may assume that there exists $\pi\in\Pi(m_X,m_Y)$ such that $\pi_n$ weakly converges to $\pi$. By Lemma \ref{lem:DisLsConti}, we have
\[
\dis_\succ S\le \liminf_{n\to\infty}\dis_\succ S_n\le\liminf_{n\to\infty}\, (t+1/n)=t.
\]
We also have
\[
\pi(S)\ge \limsup_{n\to \infty}\pi_n(S_n)\ge 1-t
\]
by Lemma \ref{lem:KP_ProkLsConti}. Then we have $X\succ_t Y$. This completes the proof.
\end{proof}

To prove Theorem \ref{thm:Lip_main} \eqref{thm:Lip_main:eq3}, we prepare some definitions. For finitely many sets $X_i$, $i=1,2,\dots,n$, and for numbers $i_1,i_2,\dots,i_k\in \{1,\dots,n\}$, let $\pr_{i_1i_2\dots i_k}\colon \prod_{i=1}^n X_i\to \prod_{j=1}^k X_{i_j}$ be the projection defined by 
\[
\pr_{i_1i_2\dots i_k}(x_1,x_2,\dots ,x_n):=(x_{i_1},x_{i_2},\dots,x_{i_k})\text{ for $x_i\in X_i,\ i=1,2,\dots,n.$}
\]
Let $X,Y,Z$ be three mm-spaces. For subsets $S\subset X\times Y$ and $T\subset Y\times Z$, we define $T\bullet S:=(X\times T)\cap (S\times Z)$, $T\circ S:=\pr_{13}(T\bullet S)$. For $\sigma\in\Pi(m_X,m_Y)$ and $\tau\in \Pi(m_Y,m_Z)$, the \textit{gluing} $\tau\bullet \sigma$ is defined by
\[
\tau\bullet \sigma(A\times B\times C):=\int_B \sigma_y(A)\tau_y(C) dm_Y(y),
\]
where $\{\sigma_y\}_{y\in Y}$ and $\{\tau_y\}_{y\in Y}$ are the disintegrations respectively of $\sigma$ and $\tau$ for the projections $\pr_2\colon X\times Y\to Y$ and $\pr_1\colon Y\times Z\to Y$. Note that the gluing $\tau\bullet\sigma$ is a Borel probability measure on $X\times Y\times Z$. It holds that $(\pr_{12})_*(\tau\bullet \sigma)=\sigma$ and $(\pr_{23})_*(\tau\bullet \sigma)=\tau$. We also define $\tau\circ \sigma:=(\pr_{13})_*(\tau\bullet\sigma)\in \Pi(m_X,m_Z)$.
\begin{proof}[Proof of Theorem \ref{thm:Lip_main} \eqref{thm:Lip_main:eq3}]
Take $X_i\in \mathcal X$ for $i=1,2,3$ and assume that $X_i\succ_{t_i} X_{i+1}$ for $i=1,2$ for non-negative real numbers $t_1$ and $t_2$. By Definition \ref{dfn:lipE}, there exist $\pi_i\in \Pi(m_{X_i},m_{X_{i+1}})$ and $S_i\subset X_i\times X_{i+1}$ such that $\dis_\succ S_i\le t_i$ and $1-\pi_i(S_i)\le t_i$ for $i=1,2$. Put 
\[
S:=\overline{S_2\circ S_1}\subset X_1\times X_3,\quad \pi:=\pi_2\circ \pi_1\in\Pi(m_{X_1},m_{X_3}).
\]
Now we have
\begin{align*}
\pi(S)&=(\pr_{13})_*(\pi_2\bullet \pi_1)(\overline{\pr_{13}(S_2\bullet S_1)})\\
&\ge \pi_2\bullet\pi_1(S_2\bullet S_1)=1-\pi_2\bullet \pi_1((X_1\times S_2)^c\cup (S_1\times X_3)^c)\\
&\ge 1-(1-\pi_2(S_2))-(1-\pi_1(S_1))=\pi_1(S_1)+\pi_2(S_2)-1\\
&\ge(1-t_1)+(1-t_2)-1=1-t_1-t_2.
\end{align*}
Let us prove that $\dis_\succ S \le t_1+t_2$. It is sufficient to prove $\dis_\succ (S_2\circ S_1)\le t_1+t_2$ because $\dis_\succ S=\dis_\succ (S_2\circ S_1)$. Take any two points $(x_1,x_3), (x_1',x_3')\in S_2\circ S_1$. There exist $x_2,x_2'\in X_2$ such that
\[
(x_1,x_2,x_3), (x_1',x_2',x_3')\in (X_1\times S_2)\cap (S_1\times X_3).
\]
We have
\begin{align*}
&\quad d_{X_3}(x_3,x_3')-d_{X_1}(x_1,x_1')\\&=(d_{X_3}(x_3,x_3')-d_{X_2}(x_2,x_2'))+(d_{X_2}(x_2,x_2')-d_{X_1}(x_1,x_1'))\\
&\le \dis_\succ S_2+\dis_\succ S_1\le t_1+t_2.
\end{align*}
This implies $X_1\succ_{t_1+t_2} X_3$. This completes the proof.
\end{proof}
\section{$1$-Lipschitz continuity up to an additive error}\label{sec:upto}
In this section, we discuss the relation between the $1$-Lipschitz continuity up to an additive error $\varepsilon$ and  $X\succ_\varepsilon Y$. 
\begin{dfn}[$3\frac 12.15$ in \cite{Gmv:green}, cf. Definition 4.38 in \cite{Shioya:mmg}]\label{dfn:upto}
Let $X$ and $Y$ be two mm-spaces. A map $f\colon \supp m_X\to \supp m_Y$ is said to be \textit{$1$-Lipschitz up to (an additive error) $\varepsilon$} if there exists a Borel subset $X_0$ of $X$ with $X_0\subset \supp m_X$ such that
\begin{enumerate}
\item $m_X(X_0)\ge 1-\varepsilon$,
\item $f^*d_Y\le d_X+\varepsilon$ on $X_0\times X_0$,
\end{enumerate}
where we define
\[
f^*d_Y(x,x'):=d_Y(f(x),f(x'))\quad\text{for $x,x'\in \supp m_X$}.
\]
\end{dfn}
\begin{dfn}[Definition 3.3 in \cite{KY:bdd1}]\label{dfn:KY}
For any mm-spaces $X$, $Y$, and any $\varepsilon\ge 0$, we write $X\succ^{KY}_\varepsilon Y$ if there exists a Borel measurable $1$-Lipschitz up to $\varepsilon$ map $f\colon \supp m_X\to \supp m_Y$ such that $d_P(f_*m_X,m_Y)\le \varepsilon$. 
\end{dfn}

The following Theorem \ref{thm:LipEqUpto} gives the relation between Definition \ref{dfn:lipE} in this paper and Definition 3.3 in \cite{KY:bdd1}.

\begin{thm}\label{thm:LipEqUpto}
Let $X$ and $Y$ be two mm-spaces and $\varepsilon$ non-negative real number. Then we have the following $(1)$ and $(2)$.
\begin{enumerate}
\item If $X\succ_\varepsilon^{KY} Y$, then we have $X\succ_{3\varepsilon} Y$.\label{thm:LipEqUpto:eq1}
\item If $X\succ_\varepsilon Y$, then for any real number $\varepsilon'>\varepsilon$, we have $X\succ_{2\varepsilon'}^{KY} Y$.\label{thm:LipEqUpto:eq2}
\end{enumerate}
\end{thm}
\begin{proof}
Let us prove \eqref{thm:LipEqUpto:eq1}. By the assumption of \eqref{thm:LipEqUpto:eq1}, there exists a Borel measurable $1$-Lipschitz up to $\varepsilon$ map $f\colon \supp m_X\to \supp m_Y$ such that $d_P(f_*m_X,m_Y)$ $\le \varepsilon$.
Since $f$ is $1$-Lipschitz up to $\varepsilon$, there exists a Borel set $X_0$ such that $m_X(X_0)\ge 1-\varepsilon$ and $d_Y(f(x),f(x'))\le d_X(x,x')+\varepsilon$ for any $x,x'\in X_0$. Put $\pi:=(\id_X,f)_*m_X\in \Pi(m_X,f_*m_X)$ and $S:=\overline{(\id_X,f)(X_0)}\in X\times Y$. Then we have $\pi(S)\ge m_X(X_0)\ge 1-\varepsilon$ and $\dis_\succ S\le \varepsilon$. These imply $(X,m_X)\succ_\varepsilon (Y,f_*m_X)$. Since $\square((Y,f_*m_X),(Y,m_Y))\le 2d_P(f_*m_X,m_Y)\le 2\varepsilon$, we have $(Y,f_*m_X)\succ_{2\varepsilon} (Y,m_Y)$ because of Proposition \ref{prop:boxToLip}. Then we have $(X,m_X)\succ_{3\varepsilon} (Y,m_Y)$ by Theorem \ref{thm:Lip_main} \eqref{thm:Lip_main:eq3}. This completes the proof of \eqref{thm:LipEqUpto:eq1}.

Next let us prove \eqref{thm:LipEqUpto:eq2}. Take any real number $\varepsilon'$ with $\varepsilon'>\varepsilon$. Put $t:=(\varepsilon'-\varepsilon)/4>0$. By Lemma 4.18 in \cite{Shioya:mmg}, there exists a finite net $\mathcal N\subset X$ such that $m_X(B_{t/2}(\mathcal N))\ge 1-t/2$. By Lemma 3.4 in \cite{Shioya:mmg}, there exists a Borel measurable nearest projection $\pi_{\mathcal N}\colon \supp m_X\to\mathcal N$. We put
\[
\dot X:=(\mathcal N, d_X, (\pi_{\mathcal N})_*m_X).
\]
By Lemma 4.19 in \cite{Shioya:mmg}, we have
\[
d_P((\pi_{\mathcal N})_*m_X,m_X)\le t/2,
\]
this implies that 
\[
\square(X,\dot X)\le 2d_P(m_X,(\pi_{\mathcal N})_*m_X)\le t<2t.
\]
Since  we have $X\succ_\varepsilon Y$ and $\square(X,\dot X)<2t$, we obtain $\dot X\succ_{\varepsilon +2t} Y$ by Theorem \ref{thm:Lip_main} \eqref{thm:Lip_main:eq3}.
By $\dot X\succ_{\varepsilon+2t} Y$, there exists a Borel set $S\subset \mathcal N\times Y$ and $\pi\in \Pi((\pi_{\mathcal N})_*m_X,m_Y)$ such that $\max\{1-\pi(S),\dis_{\succ} S\}\le \varepsilon +2t$. We may assume that $S\subset \mathcal N\times \supp m_Y$ by Proposition \ref{prop:suppPi}. Put $\dot X_0:= \pr_1(S)$. 
Then we have
\begin{align*}
m_{\dot X}(\dot X_0)&=(\pi_{\mathcal N})_*m_X(\pr_1(S))=(\pr_1)_*\pi(\pr_1(S))\\
&\ge \pi(S)\ge 1-\varepsilon-2t.
\end{align*}
By the definition of $\dot X_0$, there exists a map $\dot f\colon \dot X\to \supp m_Y$ such that $(x,\dot f(x))\in S$ for any $x\in \dot X_0$. Then we have 
\begin{equation}\label{thm:LipEqUpto:eq5}
\dot f^*d_Y\le d_X+\varepsilon+2t \quad\text{on $\dot X_0\times \dot X_0$}
\end{equation}
because $\dis_\succ S\le \varepsilon +2t$. We remark that the map $\dot f$ is measurable because $\dot X$ is finite.

Let us prove
\begin{equation}\label{thm:LipEqUpto:eq3}
d_P(\dot f_*m_{\dot X},m_Y)\le 2\varepsilon + 4t.
\end{equation}
Take any Borel set $A\subset Y$. We have
\begin{equation}\label{thm:LipEqUpto:eq4}
S\cap (\dot X_0\times A) \subset \dot f^{-1}(B_{\varepsilon+2t}(A))\times Y
\end{equation}
because we have
\[
d_Y(\dot f(x),y)\le d_X(x,x)+\dis_\succ S\le \varepsilon +2t
\]
for any $(x,y)\in S\cap (\dot X_0\times A)$. By \eqref{thm:LipEqUpto:eq4}, we have
\begin{align*}
\dot f_*m_{\dot X}(B_{\varepsilon+2t}(A))&=(\pr_1)_*\pi(\dot f^{-1}(B_{\varepsilon+2t}(A)))\\
&=\pi(\dot f^{-1}(B_{\varepsilon +2t}(A))\times Y)\\
&\ge\pi((\dot X_0\times A)\cap S)\\
&\ge \pi(\dot X_0\times A)-\varepsilon -2t\\
&=\pi((\dot X\times A)\cap(\dot X_0\times Y))-\varepsilon -2t\\
&\ge \pi(\dot X\times A)+\pi(\dot X_0\times Y)-1-\varepsilon -2t\\
&\ge \pi(\dot X\times A)-2\varepsilon -4t\\
&=m_Y(A)-2\varepsilon - 4t.
\end{align*}
This implies that \eqref{thm:LipEqUpto:eq3}.
Now, we define $f\colon \supp m_X\to \supp m_Y$ by $f:=\dot f\circ\pi_{\mathcal N}$. Let us prove that $f$ is $1$-Lipschitz up to $\varepsilon'$. Put $X_0:=\pi_{\mathcal N}^{-1}(\dot X_0)\cap B_t(\mathcal N)$.

Since $m_X(B_{t/2}(\mathcal N))\ge 1-t/2$, we have
\begin{equation}\label{thm:LipEqUpto:eq6}
m_X(X_0)\ge (\pi_{\mathcal N})_*m_X(\dot X_0)-t =m_{\dot X}(\dot X_0)-t\ge 1-\varepsilon -3t\ge 1-\varepsilon'.
\end{equation}
By \eqref{thm:LipEqUpto:eq5}, we have
\begin{equation}\label{thm:LipEqUpto:eq7}
\begin{split}
d_Y(f(x),f(x'))&=\dot f^*d_Y(\pi_{\mathcal N}(x),\pi_{\mathcal N}(x'))\\
&=d_X(\pi_{\mathcal N}(x),\pi_{\mathcal N}(x'))+\varepsilon +2t\\
&\le d_X(x,x')+d_X(x,\pi_{\mathcal N}(x))+d_X(x',\pi_{\mathcal N}(x'))\\
&\quad +\varepsilon+2t\\
&\le d_X(x,x')+\varepsilon +4t\\
&= d_X(x,x')+\varepsilon'
\end{split}
\end{equation}
for any $x,x'\in X_0$. By \eqref{thm:LipEqUpto:eq6} and \eqref{thm:LipEqUpto:eq7}, the map $f$ is $1$-Lipschitz up to $\varepsilon'$. By \eqref{thm:LipEqUpto:eq3}, we have
\[
d_P(f_*m_X,m_Y)=d_P(\dot f_*m_{\dot X},m_Y)\le 2\varepsilon +4t\le 2\varepsilon'.
\]
This completes the proof.
\end{proof}

\begin{cor}\label{cor:upto_comp}
Let $X$, $Y$, and $Z$ be three mm-spaces, and $\varepsilon_1,\varepsilon_2\ge 0$. We assume that $f\colon \supp m_X\to \supp m_Y$ is Borel measurable and $1$-Lipschitz up to $\varepsilon_1$ and $g\colon \supp m_Y\to \supp m_Z$ is Borel measurable and $1$-Lipschitz up to $\varepsilon_2$. We also assume that $d_P(f_*m_X,m_Y)\le \varepsilon_1$ and $d_P(g_*m_Y,m_Z)\le \varepsilon_2$. Then for any $s>0$, there exists a Borel measurable map $h\colon \supp m_X\to \supp m_Z$ such that $h$ is $1$-Lipschitz up to $3(\varepsilon_1+\varepsilon_2)+s$ and $d_P(h_*m_X,m_Z)\le 6(\varepsilon_1+\varepsilon_2)+s$.
\end{cor}

\begin{proof}
By Theorem \ref{thm:LipEqUpto} \eqref{thm:LipEqUpto:eq1}, we have $X\succ_{3\varepsilon_1} Y$ and $Y\succ_{3\varepsilon_2}Z$.
By Theorem \ref{thm:Lip_main} \eqref{thm:Lip_main:eq3}, $X\succ_{3(\varepsilon_1+\varepsilon_2)}Z$. Take any $s>0$. By Theorem \ref{thm:LipEqUpto} \eqref{thm:LipEqUpto:eq2}, there exists a Borel measurable map $h\colon \supp m_X\to \supp m_Z$ such that $h$ is $1$-Lipschitz up to $3(\varepsilon_1+\varepsilon_2)+s$ and $d_P(h_*m_X,m_Z)\le 6(\varepsilon_1+\varepsilon_2)+s$. This completes the proof.
\end{proof}

\begin{rem}
Note that $h=g\circ f$ does not necessarily hold in Corollary \ref{cor:upto_comp}. In fact, if we consider the following example, we  see that $g\circ f$ does not satisfy the expected conditions. Let $n\ge 3$ and $X=Y:=\{0,1/n\}$. We put $m_X:=\frac 12(\delta_0+\delta_{1/n})$, $m_Y:=(1-\frac1n)\delta_0+\frac 1n\delta_{1/n}$. We also put $Z:=n^2Y=(Y,n^2d_Y,m_Y)$ and $f=g:=\id_X$. Now, $f$ is $1$-Lipschitz and $g$ is $1$-Lipschitz up to $1/n$. Moreover, we have $d_P(f_*m_X,m_Y)\le 1/n$ and $d_P(g_*m_Y,m_Z)=0$. However, we have $\varepsilon \ge 1/2$ if $g\circ f$ is $1$-Lipschitz up to $\varepsilon$. We also have $d_P((g\circ f)_*m_X,m_Z)= 1/2-1/n$. 
\end{rem}
\begin{rem}\label{rem:lsCounter}
If we adopt Definition 3.3 in \cite{KY:bdd1} as the definition of $X\succ_\varepsilon Y$, Theorem \ref{thm:box-stable} does not hold in the case $\varepsilon> 0$. In fact, there is a counter example as follows. We denote by $*$ the one point mm-space.  We put $X_n:=*$ and $X:=*$ for any integer $n\ge 2$. We put $r:=1/4$ and
\[
Y_n:=\left(\{-r,0,r\},|\cdot|,\frac12\left(1-\frac1n\right)(\delta_{-r}+\delta_r)+\frac1n\delta_0\right)
\]
for any integer $n\ge 2$.
The sequence $(Y_n)_n$ $\square$-converges to $Y:=(\{-r,r\},|\cdot|,1/2(\delta_{-r}+\delta_r))$. We define $\square_\succ^{KY}(X,Y):=\inf\{\varepsilon\ge 0\mid X\succ_\varepsilon^{KY}Y\}$, where $X\succ^{KY}_\varepsilon Y$ is defined in Definition \ref{dfn:KY}.
Then we have $\square^{KY}_\succ(X_n,Y_n)=1/4$ and $\square^{KY}_\succ(X,Y)=1/2$ for any integer $n\ge 2$. In the case we adopt Definition \ref{dfn:lipE}, we have $\square_\succ(X_n,Y_n)=1/2(1-1/n)$ and $\square_\succ(X,Y)=1/2$ for any integer $n\ge 2$.
\end{rem}

The following Proposition \ref{prop:comp_lipu} is stronger than Corollary \ref{cor:upto_comp}. The proof of Proposition \ref{prop:comp_lipu} uses the method of the proof of Theorem 4.35 in \cite{Shioya:mmg}. Lemma 5.8 and Lemma 5.9 in \cite{KY:bdd2} are variations of Proposition \ref{prop:comp_lipu}.

\begin{prop}\label{prop:comp_lipu}
Let $X$, $Y$, and $Z$ be three mm-spaces, and $\varepsilon_1,\varepsilon_2\ge 0$. We assume that $f\colon \supp m_X\to \supp m_Y$ is Borel measurable and $1$-Lipschitz up to $\varepsilon_1$ and $g\colon \supp m_Y\to \supp m_Z$ is Borel measurable and $1$-Lipschitz up to $\varepsilon_2$. We also assume that $d_P(f_*m_X,m_Y)\le \varepsilon_1$ and $d_P(g_*m_Y,m_Z)\le \varepsilon_2$. Then for any $s>0$, there exists a Borel measurable map $h\colon \supp m_X\to \supp m_Z$ such that $h$ is $1$-Lipschitz up to $\varepsilon_1+\varepsilon_2+s$ and $d_P(h_*m_X,m_Z)\le \varepsilon_1+4\varepsilon_2+s$.
\end{prop}

\begin{proof}
Since $f$ is $1$-Lipschitz up to $\varepsilon_1$, there exists a Borel set $X_0\subset \supp m_X$ such that
\begin{equation}\label{prop:comp_lipu:eq1}
f^*d_Y\le d_X+\varepsilon_1\ \text{on $X_0\times X_0$ \quad and}\quad m_X(X_0)\ge 1-\varepsilon_1.
\end{equation}
Since $g$ is $1$-Lipschitz up to $\varepsilon_2$, there also exists a Borel set $Y_0\subset \supp m_Y$ such that 
\begin{equation}\label{prop:comp_lipu:eq2}
g^*d_Z\le d_Y+\varepsilon_2 \ \text{on $Y_0\times Y_0$ \quad and}\quad m_Y(Y_0)\ge 1-\varepsilon_2.
\end{equation}
Take any $s>0$ and put $t:=s/8$. By Lemma 3.5 in \cite{Shioya:mmg}, there exists a Borel measurable $t$-projection $\pi\colon \supp m_Y\to Y_0$ to $Y_0$ such that $\pi|_{Y_0}=\id_{Y_0}$.  Put $\tilde{X_0}:=X_0\cap f^{-1}(B_t(Y_0))$ and $h:=g\circ \pi\circ f\colon \supp m_X\to \supp m_Z$. It suffices to prove that
\begin{align}
h^*d_Z-d_X&\le \varepsilon_1+\varepsilon_2+s \quad\text{on $\tilde{X_0}\times \tilde{X_0}$},\label{prop:comp_lipu:eq3}\\
1-m_X(\tilde{X_0})&\le \varepsilon_1+\varepsilon_2+s,\label{prop:comp_lipu:eq4}\\
d_P(h_*m_X,m_Z)&\le \varepsilon_1+4\varepsilon_2+s.\label{prop:comp_lipu:eq5}
\end{align}

Let us prove \eqref{prop:comp_lipu:eq3}. Take any $x,x'\in \tilde{X_0}$. By \eqref{prop:comp_lipu:eq1}, \eqref{prop:comp_lipu:eq2}, and the definition of $\pi$, we have
\begin{align*}
d_Z(h(x),h(x'))&\le d_Y(\pi\circ f(x),\pi\circ f(x'))+\varepsilon_2,\\
&\le d_Y(f(x),f(x'))+\varepsilon_2\\
&\quad +d_Y(\pi\circ f(x),f(x))+d_Y(f(x'),\pi\circ f(x'))\\
&\le d_Y(f(x),f(x'))+\varepsilon_2\\
&\quad +d_Y(f(x),Y_0)+d_Y(f(x),Y_0)+2t\\
&\le d_Y(f(x),f(x'))+\varepsilon_2+4t\\
&\le d_X(x,x')+\varepsilon_1+\varepsilon_2+s.
\end{align*}
This implies \eqref{prop:comp_lipu:eq3}. Next let us prove \eqref{prop:comp_lipu:eq4}. Since we have $d_P(f_*m_X,m_Y)\le \varepsilon_1$, we have
\begin{align*}
m_X(\tilde{X_0})&\ge  m_X(X_0)+f_*m_X(B_t(Y_0))-1\\
&=m_Y(Y_0)-\varepsilon_1\ge 1-\varepsilon_1-\varepsilon_2.
\end{align*}
This implies \eqref{prop:comp_lipu:eq4}. Let us prepare
\begin{align}
d_P((g\circ\pi\circ f)_*m_X,(g\circ\pi)_*m_Y)&\le \varepsilon_1+2\varepsilon_2+8t,\label{prop:comp_lipu:eq6}\\
d_P((g\circ\pi)_*m_Y,g_*m_Y)&\le \varepsilon_2\label{prop:comp_lipu:eq7}
\end{align}
to prove \eqref{prop:comp_lipu:eq5}.
For $y,y'\in B_t(Y_0)$, we have
\begin{equation}\label{prop:comp_lipu:eq8}
\begin{split}
d_Z(g\circ\pi(y),g\circ\pi(y'))&\le d_Y(\pi(y),\pi(y'))+\varepsilon_2\\
&\le d_Y(y,y')+\varepsilon_2+4t.
\end{split}
\end{equation}
By Lemma 4.36 in \cite{Shioya:mmg} and \eqref{prop:comp_lipu:eq8}, we have
\begin{align*}
d_P((g\circ\pi\circ f)_*m_X,(g\circ\pi)_*m_Y)&\le d_P(f_*m_X,m_Y)+2(\varepsilon_2+4t)\\
&\le \varepsilon_1+2\varepsilon_2+8t.
\end{align*}
This implies \eqref{prop:comp_lipu:eq6}. Since we have $
d_Z(g\circ \pi(y),g(y))=0$ for any $y\in Y_0$ and $m_Y(Y_0)\ge 1-\varepsilon_2$, we have
\begin{equation}\label{prop:comp_lipu:eq9}
d_{\rm KF}^{m_Y}(g\circ\pi,g)\le \varepsilon_2.
\end{equation}
By \eqref{prop:comp_lipu:eq9} and  Lemma 1.26 in \cite{Shioya:mmg}, we obtain
\[
d_P((g\circ\pi)_*m_Y,g_*m_Y)\le d_{\rm KF}^{m_Y}(g\circ\pi,g)\le \varepsilon_2.
\]
This implies \eqref{prop:comp_lipu:eq7}. By \eqref{prop:comp_lipu:eq6}, \eqref{prop:comp_lipu:eq7}, and $d_P(g_*m_Y,m_Z)\le \varepsilon_2$, we obtain \eqref{prop:comp_lipu:eq5}. This completes the proof.
\end{proof}
  %%%% 参考文献
  
\end{document}